\documentclass[11pt,titlepage]{amsart}
\usepackage{geometry}               
\usepackage{graphicx}
\usepackage{amssymb}
\usepackage{epstopdf}

\newcommand{\R}{\mathbb{R}}

\newcommand{\C}{\mathbb{C}}

\newcommand{\D}{\mathbb{D}}
\newcommand{\Kappa}{\mathcal{K}}

\newcommand{\cis}{\operatorname{cis}}
\newcommand{\res}{\operatorname{res}}
\newcommand{\length}{\operatorname{length}}
\newcommand{\area}{\operatorname{area}}
\newcommand{\re}{\operatorname{Re}}

\newcommand{\ext}[0]{\ensuremath{\mathrm{ext}}}
\numberwithin{equation}{section}
\numberwithin{figure}{section}

\newcommand{\skp}[2]{\left({#1}\cdot{#2}\right) }
\newcommand{\sk}[2]{{#1}\cdot{#2}}
\newcommand{\dd}[1]{\frac{d}{d\, {#1}}}

\def\dnt{R^{90}}
\def\N{d}

\theoremstyle{plain} 
\newtheorem{theorem}{Theorem}[section]
\newtheorem{lemma}[theorem]{Lemma}

\newtheorem{proposition}[theorem]{Proposition}
\theoremstyle{definition} 
\newtheorem{definition}[theorem]{Definition}

\newtheorem{example}[theorem]{Example}
\title{The Gauss-Bonnet Formula for Harmonic Surfaces}

\author
{Peter Connor}
\address{Peter Connor\\Department of Mathematical Sciences\\Indiana University South Bend\\South
Bend\\IN 46634\\USA}
\author{ 
Kevin Li
}
\address{Kevin Li\\Department of Computer Science and Mathematical Sciences\\Penn State Harrisburg\\Middletown, PA 17057\\USA}

\author{
Matthias Weber
}
\address{Matthias Weber\\Department of Mathematics\\Indiana University\\
Bloomington, IN 47405
\\USA}
\thanks{This work was partially supported by a grant from the Simons Foundation (246039 to Matthias Weber)}


\date{\today}

\begin{document}

\begin{abstract}
We consider  harmonic immersions  in $\R^{\N}$ of compact Riemann surfaces with finitely many punctures where the  
harmonic coordinate functions are given as  real parts of meromorphic functions. We prove that such surfaces
have finite total Gauss curvature. The contribution of each end is a multiple of $2\pi$, determined by the maximal pole order of the meromorphic  functions. This generalizes the well known Gackstatter-Jorge-Meeks formula for minimal surfaces. The situation is complicated as the ends with their induced metrics are generally not conformally  equivalent to punctured disks, nor do the surfaces generally have limit tangent planes at the ends.
\end{abstract}

\maketitle
\section{Introduction}

The study of harmonic surfaces is largely motivated by the desire to understand to what extent this theory differs from the more special and better studied class of minimal surfaces. Several papers by Klotz \cite{kl1,kl2,kl3,kl4} from the sixties through eighties deal with the normal map of a harmonic surface and its quasiconformal properties. In a recent paper, Alarc\'o{}n and L\'o{}pez \cite{al1}  prove that a complete harmonic immersion has finite  $L^2$ norm of the shape operator ($\int |S|^2\, dA <\infty$) if and only if it satisfies Osserman's theorem in the sense that the domain of the surface is conformally a compact Riemann surface with finitely many punctures and the normal map extends continuously into the punctures.

In this paper, we will study harmonically parametrized surfaces   in $\R^{\N}$, $\N\ge 3$,  where the domain is a punctured compact
Riemann surface, and the  coordinate functions are real parts of integrals of meromorphic 1-forms. This is a larger class than that of those parametrized surfaces where merely the coordinate functions are assumed to be real parts of meromorphic functions. It contains complete minimal surfaces of finite total curvature, where in addition the sum of the squares of the meromorphic coordinate 1-forms needs to vanish to make the parametrization conformal. This relationship to minimal surfaces was our initial motivation to study this larger class of surfaces.
But, generally, in our case the parametrization is not even quasiconformal, nor does the Gauss map extend continuously into the punctures.

However, these surfaces have  finite total Gauss curvature and surprisingly satisfy a Gauss-Bonnet  formula  in the spirit of the Gackstatter-Jorge-Meeks formula \cite{gac1,jm1} for minimal surfaces. 
The proof of this formula is our main objective.

Extensions of the Gauss-Bonnet theorem to more general open surfaces have been investigated in the past: In \cite{shio1}, Shiohama 
derives a general Gauss-Bonnet formula where the contribution of the ends to the total curvature is given by a limit of circumferences of geodesic circles, {\em provided} the total curvature is finite. In   \cite{wh1}, White assumes finite $L^2$-norm of the shape operator to show that the contribution of the ends is a multiple of $2\pi$. This appears to be the first indication that this contribution is quantized under certain conditions.

We will now introduce some notations and discuss examples to explain or main theorem.
Let  $\omega_k$, $k=1,\ldots,{\N}$  be meromorphic 1-forms in the unit disk $\D$ that are holomorphic in the punctured disk $\D^*=\D-\{0\}$.
We assume that the residues $\res_0 \omega_k$ are all real. Then

\begin{equation}
\label{eqn:param}
f(z) = \re \int^z (\omega_1,\ldots,\omega_{\N})
\end{equation}
defines a harmonic map $f:\D^*\to\R^{\N}$. We say that $\D^*$ represents an {\em end}. 

Note that a regular affine transformation can change  the order of the forms $\omega_k$ while not affecting the appearance of the end by much. To obtain a rough classification of ends that is independent of affine modifications, we define:

\begin{definition}
We say that two ends $f$ and $\tilde f$  given as above are \emph{affinely equivalent} if there is a regular real affine transformation $A:\R^{\N}\to \R^{\N}$ such that $\tilde f = A\circ f$.

We say that an end is in \emph{reduced form} if the pole orders $n_k$ of $\omega_k$ at $0$ satisfy $n_1\le n_2\ldots \le n_{\N}$ and if $(n_1,n_2,\ldots,n_{\N})$ is minimal in lexicographic ordering among all affinely equivalent ends. In this case, we call 
the $\N$-tuple $(n_1,\ldots, n_{\N})$ the \emph{type} of the end.
\end{definition}

\begin{example}
The end given by
\[
f(z) =\re \int^z \left(\frac1z+\frac1{z^3},\frac1{z^2}+\frac{i}{z^3},\frac1{z^2}+\frac1{z^3}\right)\, dz
\]
has type  $(2,3,3)$ since it can be affinely transformed into
\[
\tilde f(z) =\re \int^z \left(\frac1{z}-\frac1{z^2},\frac1{z}+\frac{i}{z^3},\frac1{z^2}+\frac1{z^3}\right)\, dz \ .
\]
\end{example}

Observe that while in the domain the end looks like a punctured disk,  the Riemannian metric of the surface induced from $\R^3$ does not need to be conformally equivalent to a punctured disk:

\begin{example}
Using an extremal length argument, we will show that  with
\begin{eqnarray*}
\omega_1 &= 1\, dz
\\
\omega_2 &= \frac1z \, dz
\\
\omega_3 &= \frac i {z^2}\, dz
\end{eqnarray*}
the induced metric on $\D^*$ is not conformally equivalent to any punctured domain. 
Recall that the extremal length of a curve family $\Gamma$  on a Riemann surface is defined as
\[
\ext(\Gamma) =\sup_{\rho} \inf_{\gamma\in\Gamma} \frac {\length_\rho(\gamma)^2}{\area(\rho)}\ ,
\]
where the supremum is taken over all finite area and non-zero conformal metrics $\rho$ and $\length_\rho(\gamma)$ is the length of $\gamma$ with respect to $\rho$.
It is well known that the extremal length of the family of curves $\Gamma$ encircling a puncture is 0 for any Riemann surface. We will bound the extremal length of $\Gamma$ from below
for the conformal class of  metrics induced by the harmonic parametrization above.

To this end, we compute
in polar coordinates $z=r e^{i t}$
\[
f(r,t) = \left(r \cos(t),\log(r), -\frac1r \sin(t)\right)\ .
\]
The first fundamental form becomes $g= \frac1{r^2} g_0$ with
\[
g_0 = \begin{pmatrix}
1+r^2 \cos(t)^2 +\frac1{r^2}\sin(t)^2 & -(r^3+\frac1r)\cos(t)\sin(t)\\
-(r^3+\frac1r)\cos(t)\sin(t) & r^2 \sin(t)^2+\frac1{r^2} \cos(t)^2
\end{pmatrix} \ .
\]

Note that $g_0$ is just $g$ conformally scaled, and has finite and non-zero area. We use this metric as a test metric to estimate the extremal length of 
the curves $\Gamma$ from below:

The length of the tangent vectors to these curves  us bounded from below by  the lengths of their components in the $t$-direction. Because of
\[
g_0\left(\dd t, \dd t\right) = \cos(t)^2+r^4 \sin(t)^2 \ge \cos(t)^2\ ,
\]
the $g_0$-length of all curves enclosing  0 is bounded below by 4.
This shows that the punctured disk with metric $g_0$ (or $g$) cannot be conformally equivalent to any punctured Riemann surface.

\begin{figure}[h]
	\centerline{ 
		\includegraphics[width=2.5in]{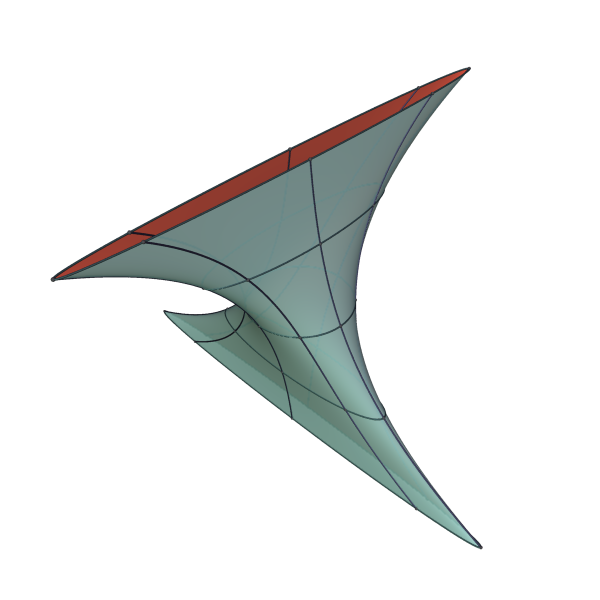}
	}
	\caption{A sphere with two ends of type $(0,1,2)$}
	\label{figure:case012}
\end{figure}

\end{example}

We will now introduce a quantity that measures the contribution of an end to the total Gauss curvature of a surface.

The total curvature of an open surface of non-positve curvature is  defined as the infimum of the total curvatures of compact subsets of the surface.
This infimum can be computed as  limit over any compact exhaustion of the surface; we choose complements of coordinate disks of shrinking radii about the punctures.

Let $X$ be a compact Riemann surface with finitely many points $p_i\in X$, and denote by $X'=X-\{p_1,\ldots,p_n\}$ the punctured surface.
Also denote by $\hat X$ the surface $X$ with disjoint  disks $D_i$ around each $p_i$ removed; this is a surface with boundary.
By the Gauss-Bonnet formula,
\[
\int_{\hat X} K \,dA+ \int_{\partial\hat X} \kappa_g ds=   2\pi \chi(\hat X) =  2\pi \chi( X')\ .
\]
Consequently,

\begin{align*}
\int_{X'} K \,dA&= \int_{\hat X} K \,dA + \sum_{i=1}^n \int_{D_i'} K \,dA\\
&=  2\pi \chi( X')  + \sum_{i=1}^n \int_{D_i'} K \,dA - \int_{\partial\hat X} \kappa_g ds\\
&=  2\pi \chi( X')  + \sum_{i=1}^n \int_{D_i'} K \,dA +\sum_{i=1}^n  \int_{\partial D_i} \kappa_g ds\ .
\end{align*}
Note that we switch the sign in the integrals over the geodesic curvature, since the boundaries of the disks $D_i$ are  the boundary components 
of $\hat X$ with opposite orientation.
This motivates the following definition:

\begin{definition}
The Gauss curvature $\Kappa_i$ of the puncture $p_i$ is the generally improper integral 
\[
\Kappa_i=   \int_{D_i'} K \,dA +\int_{\partial  D_i} \kappa_g \,ds \ .
\]
\end{definition}

Finally, we  rewrite this definition as a limit of geodesic curvature integrals. Identify the disks $D_i$ with the unit disk $|z|<1$ and 
denote by $A_i(r)$ the annulus $r<|z|<1\subset D_i$. By the Gauss-Bonnet formula again,
\begin{align*}
 \int_{D_i'} K \,dA&=\lim_{r\to 0} \int_{A_i(r)} K \,dA\\
 &=-\lim_{r\to 0}\int_{\partial  A_i(r)} \kappa_g \,ds\\
 &=\lim_{r\to 0}\int_{|z|=r} \kappa_g \,ds-\int_{\partial  D_i} \kappa_g \,ds\ .
\end{align*}

Thus we can rewrite the definition above as 

\[
\Kappa_i=   \lim_{r\to 0}\int_{|z|=r} \kappa_g \,ds \ .
\]

Observe again that this limit is independent of the chosen coordinate of $X$ about the puncture $p_i$.

The main goal of this paper is to evaluate this  limit for harmonic surfaces.

Provided we can evaluate the Gauss curvatures at the punctures, we immediately obtain a global Gauss-Bonnet theorem:

\begin{theorem}
Suppose $X$ is a compact Riemann surface with finitely many points $p_i\in X$. Let $X'=X-\{p_1,\ldots,p_n\}$ and $f:X'\to \R^{\N}$ be an immersion  with finite Gauss curvatures $\Kappa_i$ at $p_i$. Then we have
\[
\int_{X'} K\, dA - \sum_{i=1}^n \Kappa_i = 2\pi \chi(X') \ .
\]
\end{theorem}

Our main result then is

\begin{theorem}
\label{thm:kappa}
For $k\in\{1,\ldots,\N \}$, let $\omega_k$ be a meromorphic 1-form in $\D$ with pole of order $n_k$ at 0 and holomorphic elsewhere. Assume that the parametrization  
\[
f(z) = \re \int^z (\omega_1,\ldots,\omega_{\N})
\]
is an immersion. Then the Gauss curvature of this punctured end is given by
\[
\Kappa_0 = -2\pi (\max(n_k)-1) \ .
\]
\end{theorem}

In case the parametrization is conformal, i.e. when the surface is minimal, this theorem is well known \cite{gac1,jm1} and has a simple proof: It is easy to  see that the surface has a limit tangent plane at the end so that the curves $r e^{i t}$ become large circles in this tangent plane with multiplicity given by the pole order minus one. 

However, not all harmonic surfaces have limit tangent planes at their ends.
The question, then, is, why would we expect  the Gauss curvature to be quantized? We have a conjectural picture that does give an explanation, which we   illustrate with an end of type $(2,3,6)$, see Figure \ref{figure:end236}.

\begin{figure}[h]
	\centerline{ 
		\includegraphics[width=2.5in]{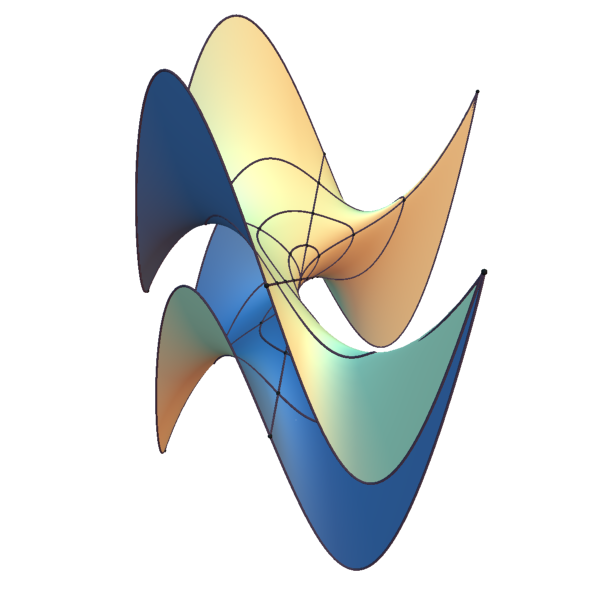}
	}
	\caption{An  embedded end of type $(2,3,6)$}
	\label{figure:end236}
\end{figure}

Numerical experiments indicate that the normal map of the surface, restricted to a curve $t\mapsto r e^{i t}$,  traces out a curve in $S^2$ that approaches the union of great circles in a single plane in $S^2$, with corners just at a pair of antipodal points. This suggests that the normal map maps the disk of radius $r$ about $0$ to a region that converges with $r\to 0$ to a union of hemispheres.

The Gauss curvature  of an end will then be the  area of this limit region, which is an integral multiple of the area of a hemisphere, see Figure \ref{figure:gauss235}.

\begin{figure}[h]
	\centerline{ 
		\includegraphics[width=2.5in]{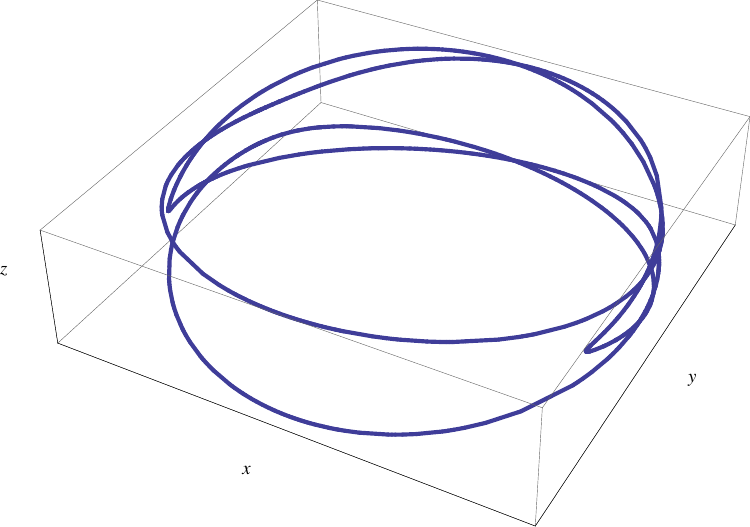}
	}
	\caption{Image of the Gauss map along $t\mapsto r e^{it}$ for the end of type $(2,3,6)$}
	\label{figure:gauss235}
\end{figure}

Our estimates are currently not strong  enough to prove a precise version of this statement. Instead we will evaluate the total curvature integral directly. This is quite delicate due to the singular nature of the integrand, shown in Figure \ref{figure:gci235} for the same end. On the other hand, our proof is essentially intrinsic, indicating that there should be a Gauss-Bonnet theorem for complete Riemannian surfaces whose ends have the same asymptotic as the ones induced by harmonic immersions of the type we consider.

\begin{figure}[h]
	\centerline{ 
		\includegraphics[width=2.5in]{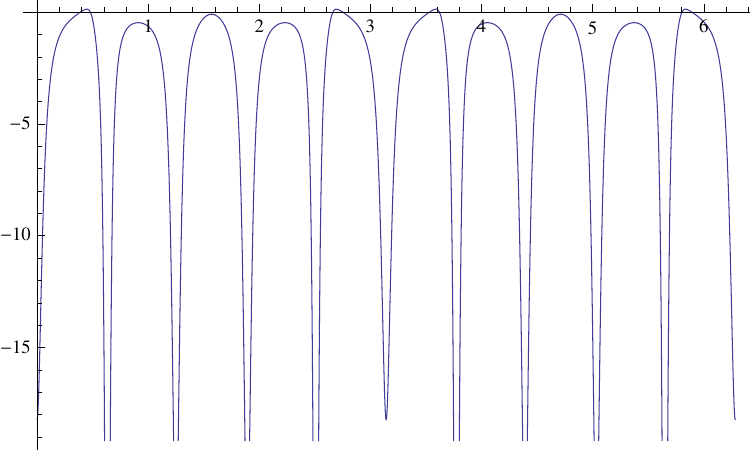}
	}
	\caption{Graph of $\eta_r(t)$ for  an  end of type $(2,3,6)$}
	\label{figure:gci235}
\end{figure}

In the example at hand, the geodesic curvature integrand $\eta_r(t) = \kappa_g(t)|c_r'(t)|$ approaches 0 in open intervals while it blows up at $t=k\pi/(\max(n_k)-1)$, where $\max(n_k) = 6$. It will turn out that this behavior is quite typical, and that each singularity will contribute the amount $-\pi$ to the total geodesic curvature when $r\to 0$.

The computations will be carried out in the  subsequent sections, which we have organized as follows:

To prove the theorem, we will distinguish two cases. By permuting the coordinates, we can assume that the pole orders are monotone: $n_1\le n_2\le \cdots\le n_{\N}$. 

In section \ref{subsec:case1}, we deal with the simplest case that $n_{{\N}-1}=n_{\N}$ where we do have a limit tangent plane at the end.

In section \ref{subsec:total}, we derive a formula for the geodesic curvature of surfaces in $\R^{\N}$ adapted to our situation and give an estimate from above.

In section \ref{subsec:normal}, we prove that we normalize the 1-form with the top pole order by a suitable holomorphic change of the coordinate.

Then, in section \ref{subsec:notation}, we introduce notation for the second case of different top two exponents $n_{\N-1}<n_{\N}$ and  $n_{\N}\ne1$. We provide a formula for the geodesic curvature in this case, expanding by powers of $r$ in polar coordinates.
This computation will reveal the singularities that the total curvature integrand develops for $r\to 0$.

We deal with these singularities using a blow-up argument in section \ref{subsec:blowup}. This requires to show that the geodesic curvature integrand is bounded uniformly by an integrable function, which is done in sections \ref{subsec:numerator} and \ref{subsec:denominator}.

Finally, in section \ref{subsec:case3} we deal with the case $n_\N=1$ which requires some special treatment.

In a logically independent companion paper \cite{clw2}, we  construct many highly complicated examples of harmonic embedded ends and complete, properly embedded harmonic surfaces to which our theorem applies. 

The authors would like to thank David Hoff, Steffen Rohde and Adam Weyhaupt for helpful discussions.

\section{Case I: Equal top exponents}\label{subsec:case1}

In this section, we will show that our main theorem holds in the case that the two top exponents are equal.
We first treat the case that the coefficients of two top exponents are independent over $\R$. This is the simplest case as here the Gauss map does still extend into the puncture.

Suppose $f:\D^*\to \R^{\N}$ is given as 
\[
f(z) =\re \int^z (\omega_1, \ldots, \omega_{\N})
\]
where $\omega_i$ are holomorphic in $\D^*$ and meromorphic in $\D$. Let $n_j$ be the order of the pole of $\omega_j$ at $0$. 

\begin{lemma}
Assume that $n_1\le n_2\le \ldots \le n_{{\N}-1} =  n_{\N}$. Write 
\[
\omega_j = (a_j z^{-n_j}+\cdots )\, dz\ ,
\]
 and assume that  $a_{{\N}-1}$ and   $a_{{\N}}$ are linearly independent over $\R$. Then $\Kappa_0 = -2\pi (n_{\N}-1)$.
\end{lemma}

\begin{proof}

For $z\in \D^*$, denote by $T_z$ the 2-dimensional tangent plane in $\R^{\N}$ of the surface $f(\D^*)$ at $f(z)$. We orient $T_z$   so that the linear maps $df|_z:\R^2\to T_z$ are orientation preserving. The map $z\mapsto T_z$ is the generalized Gauss map of $f$. We first claim that  the generalized Gauss map extends continuously to $0$.

To see this note that by our assumptions we can find a regular real linear transformation $L$ of  $\R^{\N}$  
such that $L^{-1}\circ f$ is given by meromorphic 1-forms with poles of order $n_{\N}$ only for the last two indices, and the coefficients of their leading terms are $1$ and $i$, respectively. Observe also that our assumptions force $n_d\ne 1$.

Then there is a continuous map $g:\D\to \R^{\N}$ with $g(r,t)\to 0$ for $r\to 0$ such that

\[
f(r,t) =  \frac1{r^{n_{\N}-1}} L\left( g(r,t)+ \cis((n_{\N}-1)t) \right)
\]
where 
\[
\cis(t) =  \left(0,\ldots,0,\cos(t), \sin(t) \right)\in\R^{\N}\ .
\]
Thus the limit tangent plane $T_0$ is the image of the $x_{\N-1}x_{\N}$-plane under $L$. 

Let $\Pi_z:\R^{\N}\to \R^{\N}$ the orthogonal projection onto $T_z$. 
Let $c_r(t)$ be the curve $t\mapsto f(r,t)$ for fixed $r$, and let $\gamma_r = \Pi_0\circ c_r$ be the projection onto the limit tangent plane.

By the expansion above, for small $r$ the planar curve $\gamma_r$ has winding number $n_{\N}-1$ around $0$, hence its total curvature is $-2\pi(n_{\N}-1)$, where the minus sign is dictated by our choice of orientations (it is in fact sufficient to check this in an example). 

We will now show that for $r\to 0$, the total geodesic curvature integrand of $c_r$ converges pointwise on $[0,2\pi]$  to the total curvature integrand of $\gamma_r$.

Recall that the total geodesic curvature integrand of $c_r$ is given by
\[
\eta_r(t)=\kappa_g |c_r'(t)| = \frac{\sk{c_r''}{\dnt_z c_r'}}{|c_r'|^2}
\]
where $\dnt_z$ denotes the $90^\circ$ rotation in the tangent plane $T_z$ of $f(z)$, and the  total curvature integrand of $\gamma_r$ is given by
\[
\kappa_0 |\gamma_r'(t)|=\frac{\sk{\gamma_r''}{\dnt_0 \gamma_r'}}{|\gamma_r'|^2}
\]
where $\dnt_0$ denotes the $90^\circ$ rotation in the limit tangent plane $T_0$. 

For succinctness, denote $Q_z = \dnt_z\circ \Pi_z$ the projection onto the tangent plane $T_z$ followed by the rotation in that 
tangent plane. Observe that $z\mapsto Q_z$ is continuous in all of $\D$, and in particular at $z=0$.

Using this notation, we can write
\begin{align*}
\kappa_0 |\gamma_r'(t)|&=\frac{\sk{\gamma_r''}{\dnt_0 \gamma_r'}}{|\gamma_r'|^2}
\\
&= \frac{\sk{(\Pi_0 c_r)''}{Q_0 c_r'}}{|\gamma_r'|^2}
\\
&= \frac{\sk{ c_r''}{Q_0 c_r'}}{|\gamma_r'|^2} \ .
\end{align*}

Thus

\begin{align*}
\left| \kappa_g |c_r'(t)|  - \kappa_0 |\gamma_r'(t)|\right| &= 
\left|\frac{\sk{c_r''}{\dnt_z c_r'}}{|c_r'|^2}  -  \frac{\sk{ c_r''}{Q_0 c_r'}}{|\gamma_r'|^2}\right|
\\
&= 
\left|\sk{c_r''}{\frac{Q_z c_r'}{|c_r'|^2} -  \frac{Q_0 c_r'}{|\gamma_r'|^2}}\right|
\\
&\le \frac{|c_r''|}{|c_r'|} \left|Q_z  \frac{c_r'}{|c_r'|} -   \frac{|c_r'|^2}{|\gamma_r'|^2}Q_0 \frac{c_r'}{|c_r'|}  \right|
\\
&\le \frac{|c_r''|}{|c_r'|} \left|Q_z  -   \frac{|c_r'|^2}{|\gamma_r'|^2}Q_0 \right| \ .
\end{align*}

As ${|c_r''|}/{|c_r'|}$ is bounded for $r\to 0$ and $Q_z\to Q_0$, it suffices to show that $|c_r'|/|\gamma_r'|\to 1$ for $r\to 0$.
This follows since
\[
\frac{|c_r'|}{|\gamma_r'|}=\frac{|\Pi_z c_r'|}{|\Pi_0 c_r'|}\ ,
\]
 again using that  $\Pi_z\to \Pi_0$.
\end{proof}

We now turn to the  case where still at least two top exponents are equal but all their coefficients are linearly dependent over $\R$.

Write
\[
\omega_{j} = (a_j z^{-n_\N}+ \,\textrm{lower order terms} )\, dz,
\]
where  $a_{j} \in \C$. Then  by assumption all of the (nonzero) $a_{j}$ are real multiples of each other. By making a  coordinate change of the form 
$z\mapsto \lambda z$ for suitable $\lambda\in\C$, we can assume that $a_{\N}=1$. This even holds when $n_\N=1$ as then  $a_{\N}$ has to be real anyway to make the parametrization well-defined.

Thus all of the $a_{j}$ are real, and
\begin{align*}
f(z) &= \begin{pmatrix} a_1 \\ \vdots \\a_{\N} \end{pmatrix} \re \int z^{-n_{\N}} \, dz + \textrm{lower order terms} \\
& = a \re \int z^{-n_{\N}} \, dz + w
\end{align*}
where $a \in \mathbb{R}^{\N}$ and $w$ is the vector of lower order terms.  

Let $L\in O\left(\mathbb{R}^{\N}\right)$ 
be an orthogonal transformation   that maps  $a$ to $(0,\ldots , 0, b)$ for some $b\in \R-\{0\}$.  Then
\[
L(f(z)) = \begin{pmatrix} 0 \\ \vdots \\ 0 \\ b \end{pmatrix} \re\int z^{-n_{\N}} \, dz + Lw
\]
where $Lw$ consists of strictly lower order terms (being a linear combination of lower order terms).  
Clearly, the surfaces given by $f$ and $\tilde f = L\circ f$ have the same total curvature. The forms $\tilde{\omega}_{k}$ for the surface given by $\tilde f$ will satisfy $\tilde{n}_{\N-1} < \tilde{n}_{\N}$ so that we have reduced this special case to the generic case 
that we will discuss in section \ref{subsec:notation}.

\section{Total Geodesic Curvature  for Surfaces in $\R^{\N}$}
\label{subsec:total}

In this section, we recall the formula for the geodesic curvature of surfaces in $\R^{\N}$, adapt it to our situation, and give an elementary estimate.

\begin{definition}
For vectors $X$ and $Y$ in $\R^{\N}$, denote by $X\wedge Y$ the element in the exterior product $\R^{\N}\wedge \R^{\N}$,
equipped with the norm
\[
|X\wedge Y|^2 = |X|^2|Y|^2 - \skp{X}{Y}^2\ .
\]
\end{definition}

Then we have
\begin{lemma}\label{lemma:rotate}
Let $V$ be the 2-dimensional subspace of $\R^n$ spanned by the oriented basis $X$, $Y$. Then the (oriented) 90 degree rotation in $V$ is given by
\[
\dnt U =\frac{-\skp{U}{Y}X + \skp{U}{X}Y}{|X\wedge Y|}\ .
\]
\end{lemma}

\begin{proof}
To see that $\dnt$ is a 90 degree rotation, it suffices to check that $\sk{\dnt X}{X}=0$ and $\sk{\dnt X}{\dnt X}=\sk{X}{X}$.
This is straightforward. To verify that $\dnt$ respects the orientation, it suffices to verify this for $X$ and $Y$ being orthonormal and appealing to continuity.
\end{proof}

\begin{lemma}\label{lemma:gc1}
The geodesic curvature integrand of the curve $c_r(t)=f(r,t)$ is given by 
\[
\eta_r(t)=\kappa_g |c_r'(t)|=\frac{(f_r\cdot f_t)(f_{tt}\cdot f_t)-(f_t\cdot f_t)(f_r\cdot f_{tt})}{|f_r\wedge f_t|(f_t\cdot f_t)} \ .
\]
\end{lemma}
\begin{proof}
Use Lemma \ref{lemma:rotate} in the expression for the geodesic curvature integrand
\[
\eta_r(t) = \frac{\skp{c_r''}{\dnt_z c_r'}}{|c_r'|^2}
\]
with $X=f_r$ and $U=Y=f_t$.

\end{proof}

This leads to the following upper bound for $\eta_r(t)$:
\begin{lemma}\label{lemma:gc2}
The geodesic curvature integrand of the curve $c_r(t)=f(r,t)$ has the upper bound
\[
|\eta_r(t)| \leq\frac{|f_t\wedge f_{tt}|}{|f_{t}|^2} \ .
\]
\end{lemma}
\begin{proof}
By the Binet-Cauchy identity,
\[
\begin{split}
|\eta_r(t)|&=\frac{|(f_r\cdot f_t)(f_{tt}\cdot f_t)-(f_t\cdot f_t)(f_r\cdot f_{tt})|}{|f_r\wedge f_t|(f_t\cdot f_t)}\\
&=\frac{|(f_t\wedge f_r)\cdot(f_t\wedge f_{tt})|}{|f_t\wedge f_r|(f_t\cdot f_t)}\\
&\leq\frac{|f_t\wedge f_r||f_t\wedge f_{tt}|}{|f_t\wedge f_r|(f_t\cdot f_t)}\\
&=\frac{|f_t\wedge f_{tt}|}{|f_t|^2}  \ .  \\
\end{split}
\]
\end{proof}

\section{Holomorphic change of coordinates}\label{subsec:normal}

In this section, we will normalize $\omega_{\N}$ using a holomorphic change of coordinates in $\D^*$.

For a meromorphic 1-form, the order of its pole and its residue are invariant under conformal diffeomorphisms. It is probably well known
that these are in fact the only invariants, but due to a lack of a reference, we supply a proof below. Here is the precise statement:

\begin{proposition}\label{prop:coordinates}
Let $\alpha$ and $\beta$ be meromorphic 1-forms in the unit disk, with a poles (if any) only  at the origin. Assume that the orders of $\alpha$ and $\beta$ are the same, as well as their residues. 
Then there is a holomorphic diffeomorphism $\phi$ defined in a neighborhood of the origin with $\phi(0)=0$ such that $\phi^*\alpha=\beta$.
\end{proposition}

The proposition follows from the next three lemmas, each of which treats a special case. The first two provide explicit descriptions of $\phi$, while the third uses the implicit function theorem.

\begin{lemma}
\label{lem1a}
Let $\alpha$ be a holomorphic 1-form with zero of order $n\ge0$ at 0. Then there is 
a holomorphic diffeomorphism $\phi$ near 0 such that
\[
\alpha=\phi^*\left(z^n\, dz\right) \ .
\]
\end{lemma}

\begin{proof}
The function
\[
a(z) = \int_0^z \alpha
\]
is well-defined near 0 and has a zero of order $n+1$ at 0. Thus we can solve
\[
\frac{1}{n+1} \phi(z)^{n+1} = a(z)
\]
for $\phi(z)$ as a single valued function with a simple zero near the origin.
The claim follows, as 
\[
\phi^*\left(z^n\, dz\right) = d \frac{1}{n+1} \phi(z)^{n+1}\ .
\]
\end{proof}

\begin{lemma}
\label{lem1b}
Let $\alpha$ be a meromorphic 1-form with simple pole of residue 1 at the origin. Then there is 
a holomorphic diffeomorphism $\phi$ near 0 such that
\[
\alpha=\phi^*\left(\frac1z \, dz\right)\ .
\]

\end{lemma}

\begin{proof}
Write 
\[
\alpha=\left(\frac1z+a(z)\right)\, dz
\]
with a holomorphic function $a(z)$ defined near 0, and define
\[
f(z) =e^{\int^z a(z)\,dz} \ .
\]
Then $f$ is holomorphic and non-vanishing near the origin, so 
\[
\phi(z) = z\cdot f(z)
\]
is a holomorphic diffeomorphism near the origin. Then
\[
\phi^*\left(\frac1z \, dz\right) =\left(\frac1z +\frac{f'(z)}{f(z)}\right)\, dz = \left(\frac1z +a(z)\right)\, dz 
\]
as claimed.
\end{proof}

\begin{lemma}
\label{lem1}
Let $\alpha$ be a meromorphic 1-form with pole of order $n>1$ and residue 1 at 0. Then there is a holomorphic diffeomorphism $\phi$ near 0 such that
\[
\alpha=\phi^*\left(\frac1{z^n}+\frac1z\right)\, dz \ .
\]

\end{lemma}

\begin{proof}
Write 
\[
\alpha = \left(\frac1z+\frac{h(z)}{z^n}\right)\, dz
\]
with a holomorphic function $h$ such that $h(0)\ne 0$. Note that the meromorphic form $\frac{h(z)}{z^n}\, dz$ has no residue
at 0, by assumption.

To find $\phi$, write $\phi(z)=z\cdot f(z)$ with a function $f$ to be determined that satisfies $f(0)\ne0$. It suffices to prove that
\begin{align*} 
   \left(\frac1z+\frac{h(z)}{z^n}\right)\, dz = {}& \phi^*\left(\frac1{z^n}+\frac1z\right)\, dz\\
   = {}& \phi'(z)\left(\frac1{\phi(z)^n}+\frac1{\phi(z)}\right)\, dz\\
   = {}& \left(\frac{f+z f'}{z^n f^n}+\frac1z+\frac{f'}{f}\right)\, dz \ . \\
\end{align*}

This is equivalent to 
\[
\frac{h(z)}{z^n} = \frac{f+z f'}{z^n f^n}+\frac{f'}{f} \ .
\]

Let $H(z)$ a primitive of $\frac{h(z)}{z^n}$. This is a meromorphic function of pole order $n-1$, since, as we noted, $\frac{h(z)}{z^n}\, dz$ has no residue
at 0.

Observe that the right hand side of the previous equation has an explicit primitive, it thus suffices to solve
\[
H = \frac1{1-n} \left(\frac{1}{zf}\right)^{n-1} + \log f \ .
\]
Write $f=e^F$ for some holomorphic function $F$ do be determined. This is possible as we require $f(0)\ne0$. Then we get to solve
\[
\frac1{1-n} \left(\frac{1}{ze^F}\right)^{n-1} + F = H
\]
or,
\[
\frac1{1-n}  e^{(1-n)F} + F z^{n-1}= \tilde H
\]
where $\tilde H$ is holomorphic and does not vanish near 0.
To solve this equation using the implicit function theorem write
\[
T(w,z) = \frac1{1-n}  e^{(1-n)w} + w z^{n-1}- \tilde H(z) \ .
\]
Then $T(w,0)=0$ is equivalent to 
\[
 \frac1{1-n}  e^{(1-n)w} = \tilde H(0)
\]
which has a solution $w_0$, as $\tilde H(0)\ne0$.
Furthermore,
\[
T_w(w,z) =   e^{(1-n)w} +  z^{n-1}
\]
which is non-zero for any $w$ and $z=0$. Thus there is a unique function $w= F(z)$, holomorphic at $0$, with $w_0=F(0)$ that solves
our problem.
\end{proof}

\section{Case II: Different Top Exponents, $n_\N\ne 1$ --- Notation}
\label{subsec:notation}

To evaluate the geodesic curvature integral in Lemma \ref{lemma:gc1}, we will proceed in several steps. 

We first use the normalized 1-forms to compute the integrand in polar coordinates $z=r e^{i t}$, sorted by powers of $r$ so that the  coefficients of the highest powers of $r$ are not identically vanishing in $t$.  

We will see that away from certain explicit values of $t$, this integrand converges uniformly to $0$ for $r\to 0$. At the remaining special values of $t$ where the integrand becomes singular, we use a blow-up argument $f(r, r^n t)$ with a suitable power $n$ to evaluate the limit of the total curvature integral.

As noted in section \ref{subsec:case1}, we will assume that $n_{\N-1} < n_{\N}$.  Before analyzing the geodesic curvature for this case, we illustrate the procedure with an end of type $(0,1,3)$.

\begin{example}
Consider the end of type $(0,1,3)$ with 
\[
\left(\omega_1,\omega_2,\omega_3\right)=\left(i,\frac{1}{z},\frac{1}{z^3}+\frac{1}{z^2}\right)dz \ .
\]
Analyzing the geodesic curvature is easier after normalizing the one-forms.  Applying the holomorphic diffeomorphism 
\[
\phi(z)=z^2-z\sqrt{1+z^2}
\]
to the above one-forms, we obtain the normalized one-forms
\[
\begin{split}
\psi_1&=\phi^*(\omega_1)=\left(-i+2i z-\frac{3iz^2}{2}+\frac{5iz^4}{8}+O(z)^5\right)dz\\
\psi_2&=\phi^*(\omega_2)=\left(\frac{1}{z}-1+\frac{z^2}{2}-\frac{3z^4}{8}+O(z)^5\right)dz\\
\psi_3&=\phi^*(\omega_3)=\frac{1}{z^3}dz \ . \\
\end{split}
\]
\begin{figure}[h]
	\centerline{ 
		\includegraphics[height=2in]{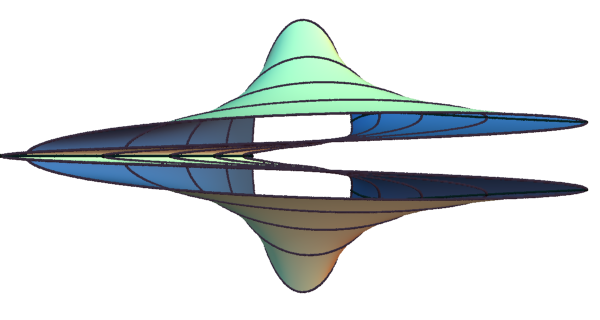}
	}
	\caption{Normalized end of type $(0,1,3)$}
	\label{figure:case013}
\end{figure}

The geodesic curvature integrand is given by 
 

\[
\begin{split}
\eta_r(t)=\kappa_g(t)|c_r'(t)|&=\frac{-r^3\left(4\cos{t}\sin^4{t}+\frac{1}{2}(1+3\cos(4t))r+O(r^2)\right)}{\sqrt{\sin^2(2t)+(-\cos{t}+\cos(5t))r+r^2+O(r^3)}\left(\sin^2(2t)+r^6+O(r^7)\right)}\\
\end{split}
\]

This is bounded for $r\to 0$ unless $\sin(2t)=0$. Away from open neighborhoods of $\sin(2t)=0$, the geodesic curvature integrand converges uniformly to 0 for $r\to 0$.  We use a blowup to evaluate the improper integral of the geodesic curvature at each singularity.

\begin{figure}[h]
	\centerline{ 
		\includegraphics[height=1.9in,width=4in]{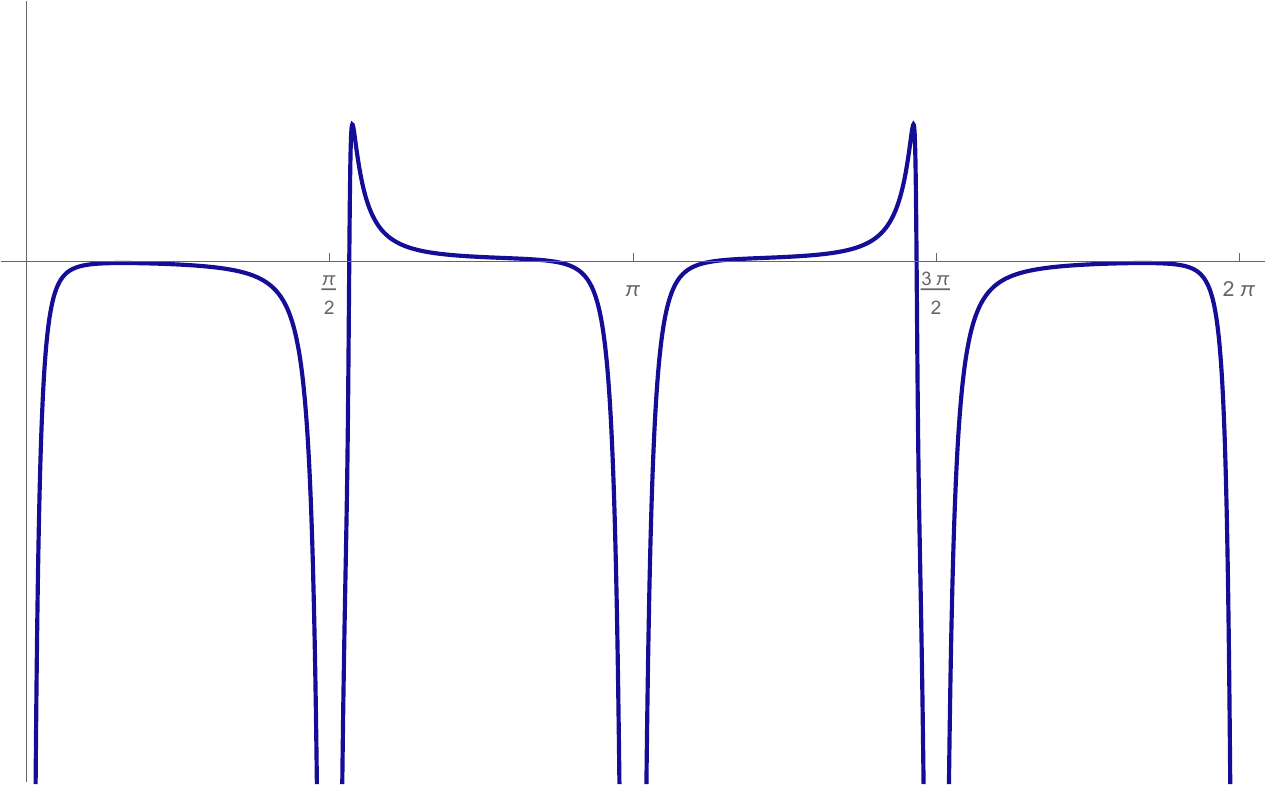}
	}
	\caption{Graph of $\eta_r(t)$ for small $r$}
	\label{figure:case013b}
\end{figure}

If $\alpha\in\{0,\pi/2,\pi,3\pi/2\}$ then 
\[
r^3\eta_r(r^3s+\alpha)=\frac{-2+O(r)}{\sqrt{1+O(r)}(4t^2+1+O(r))}
\]
and
\[
\begin{split}
\lim_{r\rightarrow 0^+}\int_0^{2\pi}\eta_r(t)dt&=\sum_{\alpha}\lim_{r\rightarrow 0^+}\int_{\alpha-\epsilon}^{\alpha+\epsilon}\eta_r(t)dt\\
&=\sum_{\alpha}\lim_{r\rightarrow 0^+}\int_{-\epsilon}^{\epsilon}r^3\eta_r(r^3 s+\alpha)ds\\
&=\sum_{\alpha}\lim_{r\rightarrow 0^+}\int_{-\infty}^{\infty}r^3\eta_r(r^3 s+\alpha)\chi_{(-\epsilon/r^3,\epsilon/r^3)}(s)ds\\
&=\sum_{\alpha}\int_{-\infty}^{\infty}\frac{-2}{4s^2+1}ds\\
&=-4\pi
\end{split}
\]
where $\displaystyle \chi_{(a,b)}$ is the characteristic function on $(a,b)$.
\end{example}

Now, consider general examples with $n_{\N-1}<n_\N$ and $n_\N\neq 1$.  Using Proposition \ref{prop:coordinates}, we can normalize the last coordinate 1-form $\omega_{\N}$.  Thus, we will assume that the ${\N}$ coordinate 1-forms are given by
\[
\begin{split}
\omega_k(z)&=\sum_{j=1}^{\infty}r_{k,j}e^{it_{k,j}}z^{j-1-n_k}dz,\hspace{.1in}1\leq k\leq {\N}-1\\
\omega_{\N}(z)&=\left(z^{-n_\N}+\frac{r_{{\N},n_{\N}}}{z}\right)dz\\
\end{split}
\]
where $n_k\leq n_{k+1}$ for $k=1,2,\ldots,{\N}-2$, $n_{{\N}-1}<n_{\N}$.

Additionally, we can assume that 
\[
h_k(z)=z^{n_k}\omega_k(z)/dz=\sum_{j=1}^{\infty}r_{k,j}e^{it_{k,j}}z^{j-1}
\]
is a convergent power series in the unit disk for $1\leq k\leq {\N}-1$.  Thus, there is a constant $M>0$ such that $r_{k,j}\leq M$ for $j\geq 1$ and $1\leq k\leq {\N}-1$. 

We also assume that $\sin(t_{k,n_k})=0$ for $k=1,2,\ldots,{\N}-1$ so that the residues of $\omega_k$ are real as required for a well-defined harmonic map in $\D^*$.  
Then we have power series expansions
\[
f(r,t)=\left(\begin{split}
&\sum_{j=1,j\neq n_1}^{\infty}r^{j-n_1}\alpha_{1,j}(t)+\alpha_{1,n_1}\log{r}\\
&\hspace{1.25in}.\\
&\hspace{1.25in}.\\
&\hspace{1.25in}.\\
&\sum_{j=1,j\neq n_{{\N}-1}}^{\infty}r^{j-n_{{\N}-1}}\alpha_{{\N}-1,j}(t)+\alpha_{{\N}-1,n_{{\N}-1}}\log{r}\\
&r^{1-n_{\N}}\beta(t)+r_{{\N},n_{\N}}\log{r}\\
\end{split}\right)
\]
where for $k=1,\ldots,{\N}-1$,
\begin{equation*}
\alpha_{k,j}(t)=
\begin{cases}
\frac{r_{k,j}\cos{((j-n_k)t+t_{k,j})}}{j-n_k},  & j\neq n_k
\\
r_{k,n_k}\cos{(t_{k,n_k})}, & j=n_k
\end{cases}
\end{equation*}
and
\[
\beta(t)=\frac{\cos((1-n_{\N})t)}{1-n_{\N}} \ .
\]
Observe that $\alpha_{k,j}''(t)=-(j-n_k)^2\alpha_{k,j}(t)$ and $\beta''(t)=-(1-n_{\N})^2\beta(t)$. 
Also note that $|\alpha_{k,j}(t)|\le M$.

Our next goal is to derive similar expansions for the relevant terms in $\eta_r(t)$ from Lemma \ref{lemma:gc1}.
To simplify the notation, we abbreviate a few sums as follows: 

\begin{definition}\label{def:STR}
For $k=1,\ldots,{\N}-1$, let 
\[
\begin{split}
S_k(r,t)&=\sum_{j=1,j\neq n_k}^{\infty}r^{j-1}\alpha_{k,j}'(t)\\
T_k(r,t)&=-S_k'(r,t)=\sum_{j=1}^{\infty}r^{j-1}(j-n_k)^2\alpha_{k,j}(t)\\
R_k(r,t)&=\sum_{j=1}^{\infty}r^{j-1}(j-n_k)\alpha_{k,j}(t)+r^{n_k-1}\alpha_{k,n_k} \ .\\
\end{split}
\]
\end{definition}
Then we have by straightforward computation:

\begin{lemma}\label{lemma:gcterms}
\[
\begin{split}
f_t&=\frac{1}{r^{n_{\N}-1}}\left(r^{n_{\N}-n_1}S_1(r,t),\ldots,r^{n_{\N}-n_{{\N}-1}}S_{{\N}-1}(r,t),\beta'(t)\right)\\
f_{tt}&=-\frac{1}{r^{n_{\N}-1}}\left(r^{n_{\N}-n_1}T_1(r,t),\ldots,r^{n_{\N}-n_{{\N}-1}}T_{{\N}-1}(r,t),(1-n_{\N})^2\beta(t)\right)\\
f_r&=\frac{1}{r^{n_{\N}}}\left(r^{n_{\N}-n_1}R_1(r,t),\ldots,r^{n_{\N}-n_{{\N}-1}}R_{{\N}-1}(r,t),(1-n_{\N})\beta(t)+r^{n_{\N}-1}r_{{\N},n_{\N}}\right) \ .\\
\end{split}
\]
\end{lemma}

\section{The Blow-Up Argument}\label{subsec:blowup}

Recall from Lemma \ref{lemma:gc2} that the geodesic curvature integrand is bounded above by
\[
|\eta_r(t)| \le  \frac{|f_t\wedge f_{tt}|}{ |f_t|^2}\ .
\]
By Lemma \ref{lemma:gcterms}, this is bounded for $r\to 0$ unless $\beta'(t)=0$. Moreover, the formula for $\eta_r(t)$ in Lemma \ref{lemma:gc1} together with Lemma \ref{lemma:gcterms} shows that away from open neighborhoods of $\beta'(t)=0$, the integrand converges uniformly to 0 for $r\to 0$. 

We will now analyze the singular behavior of the geodesic curvature integrand.
Suppose $\beta'(t_0)^2=0$.  Then $\beta(t_0)^2=\frac{1}{(1-n_3)^2}\neq 0$, and again  Lemma \ref{lemma:gc1}  and Lemma \ref{lemma:gcterms} imply that  our integrand does indeed have a singularity at $t_0$ when $r\to 0$.
We will cope with these singularities using a blow-up argument.

Since $\beta'(t)=-\sin((1-n_{\N})t)$, these singularities are explicit. It will turn out that  in our choice of coordinate they all will contribute the same amount to the total curvature integral. 
We will consider the case $t_0=0$. The other cases are notationally more complicate but are treated the  same way.

The idea is to make a substitution of the form $(r,t)\mapsto (r,r^n t)$ for a suitable exponent $n$. Our next goal is to determine that exponent by the power of $r$  by which the $S_k$ terms  possibly converge to 0 when $r\to 0$.

\begin{definition}\label{def:mk}
Let $1\leq k\leq {\N}-1$.
\noindent
Define $m_k$ as the first integer such that 

$r_{k,m_k}\sin(t_{k,m_k})\neq 0$.  That is, $m_k-1$ is the order of the zero of $S_k(r,0)$ at $r=0$.  
\end{definition}
Then we expand

\begin{equation}
\label{eqn:Sk}
\begin{split}
S_k(r,t)=&-\sum_{j=1,j\neq n_k}^{m_k-1}r^{j-1}r_{k,j}\sin{((j-n_k)t)}\cos(t_{k,j})\\
&-r^{m_k-1}\sum_{j=m_k,j\neq n_k}^{\infty}r^{j-m_k}r_{k,j}\left[\sin{(j-n_k)t)}\cos{(t_{k,j})}+\cos{((j-n_k)t)}\sin{(t_{k,j})}\right] \ .\\
\end{split}
\end{equation}
If $m_k=\infty$ then 
\begin{equation}
\label{eqn:Sk2}
S_k(r,t)=-\sum_{j=1,j\neq n_k}^{\infty}r^{j-1}r_{k,j}\sin{((j-n_k)t)}\cos(t_{k,j}) \ .\\
\end{equation}
This allows us to determine the critical exponent $n$ as well as to introduce abbreviations which will be used in our estimate of the geodesic curvature integrand:
 
\[
\begin{split}
n&=\min_k\{n_\N-n_k+m_k-1\}\\
k^*&=\{k\in\{1,\ldots,{\N}-1\}:n_\N-n_k+m_k-1=n\}\\
a&=\frac{n_{\N}-1}{\sqrt{b}}\\
\end{split}
\]
with 
\[
b=\sum_{k\in k^*}r_{k,m_k}^2\sin^2(t_{k,m_k}) \ .
\]
Note that as $f$ is assumed to be an immersion, not all of the $m_k$ can be infinite, so that $n$ is finite.

\begin{proposition}
\label{prop:blowup} 
With the notation introduced above,
\[
\begin{split}
\lim_{r\rightarrow 0}r^n \eta_r(r^nt)&=-\frac{a}{1+a^2t^2} \ .\\
\end{split}
\]
\end{proposition}
\begin{proof}
The terms from Lemma \ref{lemma:gcterms} can be estimated using
\[
\begin{split}
S_k(r,r^nt)&=
\begin{cases}
-r^{m_k-1}r_{k,m_k}\sin{(t_{k,m_k})}+O(r^{m_k}), \hspace{.15in}m_k<\infty
\\
r_{k,1}(n_k-1)\cos(t_{k,1})r^nt+O(r^{n+1}), \hspace{.1in}m_k=\infty
\end{cases}\\
\beta'(r^nt)&=r^{n}t(n_{\N}-1)+O(r^{n+1})\\
T_k(r,r^nt)&=(1-n_k)r_{k,1}\cos(t_{k,1})+O(r)\\
R_k(r,r^nt)&=r_{k,1}\cos(t_{k,1})+O(r)\\
\beta(r^nt)&=\frac{1}{1-n_{\N}}+O(r)\ . \\
\end{split}
\]
Then, by straightforward computation,

\begin{align*}
f_t\cdot f_t&=
\frac{r^{2n}}{r^{2(n_{\N}-1)}}\left[b+(1-n_{\N})^2t^2+O(r)\right]
\\
f_r\cdot f_{tt}&=-\frac{1}{r^{2n_{\N}-1}}\left[1-n_{\N}+O(r)\right]
\\
f_r\cdot f_t&=
-\frac{r^n}{r^{2n_{\N}-1}}\left[(1-n_{\N})t+O(r)\right]
\\
f_{tt}\cdot f_t&=
\frac{r^n}{r^{2(n_{\N}-1)}}\left[(1-n_{\N})^2t+O(r)\right]
\\
f_r\cdot f_r&=
\frac{1}{r^{2n_{\N}}}\left[1+O(r)\right] \ .
\\
\end{align*}

Combining everything gives
\[
\begin{split}
(f_r\cdot f_t)(f_{tt}\cdot f_t)-(f_t\cdot f_t)(f_r\cdot f_{tt})&=\frac{r^{2n}}{r^{4n_{\N}-3}}\left[(1-n_{\N})b+O(r)\right]\\\\
(f_r\cdot f_r)(f_t\cdot f_t)-(f_r\cdot f_t)^2&=\frac{r^{2n}}{r^{4n_{\N}-2}}\left[b+O(r)\right]\\
\end{split}
\]

and so
\[
\begin{split}
r^n\eta_r(r^nt)&=\frac{r^n\frac{r^{2n}}{r^{4n_{\N}-3}}\left[(1-n_{\N})b+O(r)\right]}{\sqrt{\frac{r^{2n}}{r^{4n_{\N}-2}}\left[b+O(r)\right]}\frac{r^{2n}}{r^{2(n_{\N}-1)}}\left[b+(1-n_{\N})^2t^2+O(r)\right]}\\
&=\frac{(1-n_{\N})b+O(r)}{\sqrt{b+O(r)}\left[b+(1-n_{\N})^2t^2+O(r)\right]}\\
&=\frac{(-a+O(r))}{\sqrt{1+O(r)}\left[1+a^2t^2+O(r)\right]} \ . \\
\end{split}
\]
Hence,
\[
\lim_{r\rightarrow 0}r^n\eta_r(r^nt)=-\frac{a}{1+a^2t^2}
\]
as claimed.
\end{proof}

To finish the proof of Theorem \ref{thm:kappa}, we will need to use that that total curvature integrand of an end is 
bounded by an integrable function. This is accomplished below. 

\begin{lemma}\label{lemma:estimate}
Let $\displaystyle\epsilon<\min\left\{1,\frac{\pi}{2(n_{\N}-1)}\right\}$ and small enough such that Lemma \ref{lemma:denom} can be applied. Then there is a constant $M$ depending only on the forms $\omega_k$ such that with  
\[
g(s)=\frac{M}{1+s^2} \left(|s|^{\frac{n-n_{\N}+n_{{\N}-1}}{n}}+1\right)
\]
we have
\[
r^n\left|\eta_r(r^ns)\right|\chi_{(-\epsilon/r^n,\epsilon/r^n)}(s)<g(s)
\]
for all $r<1$ and $s\in\R$. Here $\chi_{(a,b)}$ denotes the characteristic function of the interval $[a,b]$.
\end{lemma}
\begin{proof}
We combine Lemma \ref{lemma:num} and Lemma \ref{lemma:denom}, which are proven in sections \ref{subsec:numerator} and \ref{subsec:denominator}, where we deal with the numerator and denominator of the geodesic curvature integrand separately.
They yield that there is a constant $M$ such that for small $r$,

\[
|\eta_r(t)|\leq\frac{M(|t|r^{n_{\N}-n_{{\N}-1}}+r^n)}{r^{2n}+t^2} \ .
\]
If $n_{{\N}-1}<n_{\N}$ and $|t|<\epsilon<1$ then 
\[
|t|<|t|^{\frac{n-n_{\N}+n_{{\N}-1}}{n}}
\]
and so
\[
|\eta_r(t)|\leq\frac{M(|t|^{\frac{n-n_{\N}+n_{{\N}-1}}{n}}r^{n_{\N}-n_{{\N}-1}}+r^n)}{r^{2n}+t^2} \ .
\]
Substituting $r^ns<\epsilon$ for $t$,
\[
\begin{split}
r^n\left|\eta_r(r^ns)\right|&\leq\frac{r^nM\left(|r^ns|^{\frac{n-n_{\N}+n_{{\N}-1}}{n}}r^{n_{\N}-n_{{\N}-1}}+r^n\right)}{r^{2n}+(r^ns)^2}\\
&\leq \frac{r^{2n}M\left(|s|^{\frac{n-n_{\N}+n_{{\N}-1}}{n}}+1\right)}{r^{2n}\left(1+s^2\right)}\\
&\leq \frac{M\left(|s|^{\frac{n-n_{\N}+n_{{\N}-1}}{n}}+1\right)}{1+s^2}\\
\end{split}
\]
which is integrable on $\R$.  With
\[
g(s)=\frac{M\left(|s|^{\frac{n-n_{\N}+n_{{\N}-1}}{n}}+1\right)}{1+s^2}
\]
we have
\[
r^n\left|\eta_r(r^ns)\right|\chi_{(-\epsilon/r^n,\epsilon/r^n)}(s)<g(s)
\]
for all $r<1$ and $s\in\R$, as claimed.  

\end{proof}

 This given, we can now prove:

\begin{proposition}
Let $\displaystyle\epsilon$ be as in Lemma \ref{lemma:estimate}. Then
\[
\lim_{r\rightarrow 0^+}\int_{-\epsilon}^{\epsilon}\eta_r(t)\, dt=-\pi \ .
\]
\end{proposition}
\begin{proof}
Using the blow-up substitution $t=r^ns$, we obtain
\[
\begin{split}
\int_{-\epsilon}^{\epsilon}\eta_r(t)\, dt&=\int_{-\epsilon/r^{n}}^{\epsilon/r^{n}}r^{n}\eta_r(r^{n}s)\,ds\\
&=\int_{-\infty}^{\infty}r^{n}\eta_r(r^ns)\chi_{(-\epsilon/r^n,\epsilon/r^n)}(s)\,ds \ .\\
\end{split}
\]

By Lemma \ref{lemma:estimate},  the estimate
\[
r^n\left|\eta_r(r^ns)\right|\chi_{(-\epsilon/r^n,\epsilon/r^n)}(s)<g(s)
\]
holds for all $r<1$ and $s\in\R$.  Hence, by Proposition \ref{prop:blowup} and the dominated convergence theorem,
\[
\begin{split}
\lim_{r\rightarrow 0^+}\int_{-\epsilon}^{\epsilon}\eta_r(t)\, dt&=\int_{-\infty}^{\infty}\lim_{r\rightarrow 0^+}r^n\eta_r(r^ns)\chi_{(-\epsilon/r^n,\epsilon/r^n)}(s)\,ds\\
&=-\int_{-\infty}^{\infty}\frac{a}{1+(as)^2}\,ds\\
&=-\pi \ .
\end{split}
\]
\end{proof}

\section{Numerator Estimate}\label{subsec:numerator}

The purpose of this and the following section is to prove the integrability Lemma \ref{lemma:estimate}.
In this section, we will estimate the numerator of the geodesic curvature from above.

We begin by providing estimates for the individual terms in the numerator:

\begin{lemma}
\label{lemma:STR}
 There is a constant $M$ such that for $r<1$ we have

\begin{align*}
|S_k(r,t)|&\leq M(|\sin{t}|+r^{m_k-1})
\\
|T_k(r,t)|&\leq M
\\
|\beta(t)|&\leq\frac{1}{n_{\N}-1}
\\
|\beta'(t)|&\leq (n_{\N}-1)|\sin{t}| \ .
\end{align*}
\end{lemma}
\begin{proof}
We choose $M>r_{k,j}$ as before  and will absorb any constant terms into $M$ as well.

Using equation \ref{eqn:Sk}, we obtain

\[
\begin{split}
|S_k(r,t)|&\leq \sum_{j=1}^{m_k-1}M r^{j-1}|\sin(t(n_k-j)|)+r^{m_k-1}\sum_{j=m_1}^{\infty}M r^{j-m_k}\\
&\leq\frac{M(1-r^{m_k-1})|\sin{t}|}{1-r}+\frac{Mr^{m_k-1}}{1-r} \\
&\leq M(|\sin{t}|+r^{m_k-1}) \ .\\
\end{split}
\]

From Definition \ref{def:STR} we get
\[
\begin{split}
|T_k(r,t)|&\leq\sum_{j=1}^{\infty}M|j-n_k|r^{j-1}\\
&\leq\sum_{j=1}^{\infty}Mn_kr^{j-1}+\sum_{j=1}^{\infty}Mjr^{j-1}\\
&\leq\frac{Mn_k}{1-r}+\frac{M}{(1-r)^2}\\
&\leq M \ .
\end{split}
\]

The bounds for 
$|\beta(t)|$ and $|\beta'(t)|$ are immediate.
\end{proof}

The main estimate of this section is contained in
\begin{lemma}
\label{lemma:num}
If $r<1$ then
\[
r^{2(n_{\N}-1)}|f_t\wedge f_{tt}|\leq M(|t|r^{n_{\N}-n_{{\N}-1}}+r^n)
\]
where $M$ is a constant depending only on the $\omega_k$.
\end{lemma}
\begin{proof}
Using Lemma \ref{lemma:gcterms}, a simple calculation produces the following expression for the left hand side 
in the claim of Lemma \ref{lemma:num}:
\[
\begin{split}
&r^{4(n_{\N}-1)}\left(|f_t|^2|f_{tt}|^2-(f_t\cdot f_{tt})^2\right)\\
=&\left(\sum_{k=1}^{{\N}-1}r^{2(n_{\N}-n_k)}S_k^2+\beta'(t)^2\right)\left(\sum_{k=1}^{{\N}-1}r^{2(n_{\N}-n_k)}T_k^2+(1-n_{\N})^4\beta(t)^2\right)\\
&-\left(\sum_{k=1}^{{\N}-1}r^{2(n_{\N}-n_k)}S_kT_k+(1-n_{\N})^2\beta(t)\beta'(t)\right)^2\\
=&\sum_{k<j}r^{2(2n_{\N}-n_k-n_j)}\left(S_k^2T_j^2-2S_kT_kS_jT_j+S_j^2T_k^2\right)\\
&+\sum_{k=1}^{{\N}-1}r^{2(n_{\N}-n_k)}\left((1-n_{\N})^4\beta(t)^2S_k^2+\beta'(t)^2T_k^2-2S_kT_k(1-n_{\N})^2\beta(t)\beta'(t)\right)\\
=&\sum_{k<j}r^{2(2n_{\N}-n_k-n_j)}(S_kT_j-S_jT_k)^2+\sum_{k=1}^{{\N}-1}r^{2(n_{\N}-n_k)}\left((1-n_{\N})^2\beta(t)S_k-\beta'(t)T_k\right)^2 \ .\\
\end{split}
\]

Thus we obtain  the following upper bound:
\[
\begin{split}
&r^{2(n_{\N}-1)}\sqrt{|f_t|^2|f_{tt}|^2-(f_t\cdot f_{tt})^2}\\
\leq&\sum_{k<j}r^{2n_{\N}-n_k-n_j}\left(|S_k||T_j|+|S_j||T_k|\right)+\sum_{k=1}^{{\N}-1}r^{n_{\N}-n_k}\left((1-n_{\N})^2|\beta(t)||S_k|+|\beta'(t)||T_k|\right) \ .\\
\end{split}
\]

Then, by Lemma \ref{lemma:STR},
\[
\begin{split}
r^{2n_{\N}-n_k-n_j}\left(|S_k||T_j|+|S_j||T_k|\right)&\leq M\left(r^{2n_{\N}-n_k-n_j}|\sin{t}|+r^{2n_{\N}-n_k-n_j+m_k-1}+r^{2n_{\N}-n_k-n_j+m_j-1}\right)\\
&\leq M\left(r^{n_{\N}-n_{{\N}-1}}r^{n_{\N}+n_{{\N}-1}-n_k-n_j}|\sin{t}|+r^n(r^{n_{\N}-n_j}+r^{n_{\N}-n_k})\right)\\
&\leq M\left(r^{n_{\N}-n_{{\N}-1}}|t|+r^n\right)\\
\end{split}
\]
and
\[
\begin{split}
r^{n_{\N}-n_k}\left((1-n_{\N})^2|\beta(t)||S_k|+|\beta'(t)||T_k|\right)&\leq r^{n_{\N}-n_k}\left((n_{\N}-1)M(|\sin{t}|+r^{m_k-1})+M(n_{\N}-1)|\sin{t}|\right)\\
&\leq M\left(r^{n_{\N}-n_k}|\sin{t}|+r^{n_{\N}-n_k+m_k-1}\right)\\
&\leq M(r^{n_{\N}-n_{{\N}-1}}|t|+r^n) \ .\\
\end{split}
\]
Thus,
\[
\begin{split}
r^{2(n_{\N}-1)}\sqrt{|f_t|^2|f_{tt}|^2-(f_t\cdot f_{tt})^2}\leq &\sum_{k<j}M\left(r^{n_{\N}-n_{{\N}-1}}|t|+r^n\right)+\sum_{k=1}^{{\N}-1}M\left(r^{n_{\N}-n_{{\N}-1}}|t|+r^n\right)\\
\leq&M\left(|t|r^{n_{\N}-n_{{\N}-1}}+r^n\right)\\
\end{split}
\]
which proves the claim.
\end{proof}

\section{Denominator estimate}\label{subsec:denominator}

In this section, we will prove the following  lower bound for the denominator term $|f_t|$ of the geodesic curvature:

\begin{lemma}\label{lemma:denom}
There is a constant $\mu>0$ such that for $|t|$ and  $r$  small enough we have
\[
 r^{2(n_{{\N}}-1)} |f_{t}|^2  \ge\mu \left( r^{2n} +  t^2\right) \ .
\]

\end{lemma}

We begin the proof by setting up the terms that need to be estimated.
Recall from Lemma \ref{lemma:gcterms} that

\[f_{t} = \frac{1}{r^{n_{{\N}}-1}} \left( r^{n_{{\N}} - n_{1}} S_{1} (r,t) , \dots, r^{n_{{\N}}-n_{{\N}-1}} S_{{\N}-1} (r,t), \beta'(t) \right)\]
where
\[
\begin{split}
S_{k}(r,t) =& -\sum_{j=1,j\neq n_k}^{m_k-1} r^{j-1} r_{k,j}\sin((j-n_k)t)\cos(t_{k,j})
\\
& -\sum_{j=m_k, j\neq n_k}^{\infty} r^{j-1} r_{k,j} \left(\sin((j-n_k)t)\cos(t_{k,j})+\cos((j-n_k)t)\sin(t_{k,j})\right)
\\
\end{split}
\]
and
\[
\beta(t) = \frac{\cos((1-n_{{\N}})t)}{1-n_{{\N}}} \ .
\]

Abbreviate

\begin{align*} 
A_k & = \sum_{j=2,j\neq n_k}^{\infty} -r_{k,j} \sin ((j-n_{k})t) \cos(t_{k,j})r^{j-1}\\
& = \sum_{j=2,j\neq n_k}^{m_k}-r_{k,j} \sin ((j-n_{k})t)\cos(t_{k,j}) r^{j-1} +  \sum_{j=m_k+1,j\neq n_k}^{\infty} -r_{k,j} \sin ((j-n_{k})t) \cos(t_{k,j}) r^{j-1}
\\
B_k &= \sum_{j=m_k+1,j\neq n_k}^{\infty} -r_{k,j} \cos ((j-n_{k})t)\sin (t_{k,j}) r^{j-1}.
\end{align*}

Notice that by  Definition \ref{def:mk} of the numbers $m_{k}$,
\[
A_k+B_k=r^{n_{k} - 1}\left( f_{t} \right)_{k} - \left(-r_{k,1} \sin ((1-n_{k})t)\cos(t_{k,1})-r_{k,m_k}\cos((m_k-n_k)t)\sin(t_{k,m_k})r^{m_k-1}\right) \ .
\]

Generically, we expect $r_{k,1} \sin(t_{k,1}) = 0$, which is what we will assume for the rest of the proof.
  If $r_{k,1} \sin(t_{k,1}) \neq 0$, then $m_{k} = 1$, and  the first sum in $A_k$ is empty.  This will mean that there is no $|t|$ term in the estimates below, simplifying the argument.

Our first goal is to estimate the components of $f_{t}$ from above and below:

\begin{lemma}\label{lemma:upperlower}
There are nonzero constants $C_{k},D_{k},C'_{k},D'_{k}$   that depend on  $r_{k,1}$ and $r_{k,m_{k}} \sin (t_{k,m_{k}})$ such that $C_{k}>0$,$C'_{k}<0$ and $D_{k}$ and $D'_{k}$ have the same sign and so that
\[
C'_{k}|t| + D'_{k} r^{m_{k}-1} \le r^{n_{k} - 1} \left(f_{t} \right)_{k} \le C_{k} |t| + D_{k} r^{m_{k}-1}\ .
\]

\end{lemma}
\begin{proof}
We begin by showing that $A_k$ and $B_k$ are relatively small:
\[\begin{split}
|A_k| & \leq \sum_{j=2,j\neq n_k}^{m_k}M|j-n_{k}| |t| r^{j-1} +  \sum_{j=m_k+1,j\neq n_k}^{\infty} M r^{j-1} \\
& \leq MC \frac{r}{1-r}|t| + M \frac{r^{m_{k}}}{1-r} \\
& < \delta' \left( |t| + r^{m_{k}-1} \right)
\end{split}\]
for $r$ small enough. Secondly,

\[\begin{split}
|B_k| & \leq \sum_{j=m_k+1,j\neq n_k}^{\infty}M  r^{j-1} \\
& = M \frac{r^{m_{k}}}{1-r} \\
& < \delta' r^{m_{k}-1} \ .
\end{split}\]
Thus,
\[
\begin{split} 
r^{n_{k} - 1} \left(f_{t}\right)_{k} &\leq -r_{k,1}\sin ((1-n_{k})t)\cos(t_{k,1}) -r_{k,m_{k}} \cos ((m_{k}-n_{k})t)\sin(t_{k,m_{k}}) r^{m_{k}-1} 
\\
&\qquad + \delta' |t| + 2\delta' r^{m_{k}-1} \\
& \leq C_{k} |t| + D_{k} r^{m_{k}-1}
\end{split}
\]
and
\[
\begin{split} 
r^{n_{1}-1}\left(f_{t} \right)_{k} & \geq -r_{k,1}\sin ((1-n_{k})t)\cos(t_{k,1}) -r_{k,m_{k}} \cos ((m_{k}-n_{k})t)\sin(t_{k,m_{k}}) r^{m_{k}-1} 
\\
&\qquad - \delta' |t| - 2\delta' r^{m_{k}-1} \\
& \geq C'_{k}|t| + D'_{k} r^{m_{k}-1} \ .
\end{split}
\]

\end{proof}

Since we really want to estimate $|\left(f_{t}\right)_{k}|^2$, we will need the following simple consequence of Young's inequality (see \cite{hlp}) which we state without proof.

\begin{lemma}\label{lemma:young}
For any nonzero constants $C, D, N,\xi$, the following inequality holds for all  $t$ and all $r> 0$:
\[
|Ct + Dr^N|^2 \geq \left( 1- \frac{1}{\xi^2}\right) C^2 t^2 + \left( 1- \xi^2 \right) D^2 r^{2N} \ .
\]

\end{lemma}

Now we apply this lemma to obtain the desired bound for all indices except $k=\N$.

\begin{lemma}
For any $0 < |\xi| < 1$, there exists  nonzero constants $C''$ and $D''$ such that
\[r^{2(n_{k} - 1)}\left|\left(f_{t} \right)_{k} \right|^2 \geq \left( 1- \frac{1}{\xi^2}\right) C''^2 t^2 + \left( 1- \xi^2 \right) D''^2 r^{2(m_{k}-1)}
\]
for  $k = 1,2, \dots, \N-1$
\end{lemma}

\begin{proof}
 We first consider the case when $D_{k}, D'_{k}>0$.  Let 
 \[
 \Omega = \{ (r,t) : C'_{k}|t| + D'_{k}r^{m_{k}-1} < 0\}\ ,
 \]
 and denote by $\Omega^c$ the complement.
 
By Lemma \ref{lemma:upperlower} we have  on $\Omega^c$
\[
r^{n_{k} - 1}\left(f_{t} \right)_{k} \geq C'_{k}|t| + D'_{k}r^{m_{k}-1} \geq 0.
\]  
Thus Lemma \ref{lemma:young}  implies
\[r^{2(n_{k} - 1)}\left| \left(f_{t} \right)_{k}\right|^2 \geq \left( 1- \frac{1}{\xi^2}\right) C_{k}'^2 t^2 + \left( 1- \xi^2 \right) D_{k}'^2 r^{2(m_{k}-1)}.\]
Now, on $\Omega$, we have $0 \leq D_{k}'r^{m_{k}-1} < -C_{k}'|t|$ which implies that
\[
D_{k}'^2 r^{2(m_{k}-1)} < C_{k}'^2 |t|^2\ .
\]
As 
\[
0 \leq 1 - \xi^2 < \frac{1}{\xi^2} - 1
\]
for all nonzero $\xi \neq \pm 1$ we obtain
\[
\left( 1- \xi^2 \right) D_{k}'^2 r^{2(m_{k}-1)} < \left( \frac{1}{\xi^2}-1\right) C_{k}'^2 t^2
\]
which  is equivalent to our claim
\[
\left( 1- \frac{1}{\xi^2}\right) C_{k}'^2 t^2 + \left( 1- \xi^2 \right) D_{k}'^2 r^{2(m_{k}-1)} < 0\ .
\]
 
 Therefore,
\[r^{2(n_{k}-1)}\left| \left(f_{t} \right)_{k} \right|^2 \geq 0 >\left( 1- \frac{1}{\xi^2}\right) C_{k}'^2 t^2 + \left( 1- \xi^2 \right) D_{k}'^2 r^{2(m_{k}-1)}.\]
This proves the result when $D_{k},D_{k}'>0$.

Next we consider the case when $D_{k},D_{k}'<0$.  
Let 
\[
\Omega = \{(r,t) : C_{k}|t| + D_{k}r^{m_{k}-1} > 0\}\ .
\]

By Lemma \ref{lemma:upperlower} we have  on $\Omega^c$

\[
r^{n_{k}-1}\left(f_{t} \right)_{k} \leq C_{k}|t| + D_{k}r^{m_{k}-1} \leq 0 \ ,
\]  
which gives the trivial lower bound
\[
r^{2(n_{k}-1)}\left| \left(f_{t} \right)_{k} \right|^2 \geq \left( 1- \frac{1}{\xi^2}\right) C_{k}^2 t^2 + \left( 1- \xi^2 \right) D_{k}^2 r^{2(m_{k}-1)}.\]
Now, on $\Omega$, we have $0 \leq -D_{k}r^{m_{k}-1}< C_{k}|t|$  and so
\[D_{k}^2 r^{2(m_{k}-1)} < C_{k}^2 |t|^2.\]
As above, this implies
\[
\left( 1- \frac{1}{\xi^2}\right) C_{k}^2 t^2 + \left( 1- \xi^2 \right) D_{k}^2 r^{2(m_{k}-1)} < 0
\]
as long as   $0 < |\xi| < 1 $.  
Therefore,
\[
r^{2(n_{k}-1)}\left| \left(f_{t} \right)_{k} \right|^2 \geq 0 >\left( 1- \frac{1}{\xi^2}\right) C_{k}^2 t^2 + \left( 1- \xi^2 \right) D_{k}^2 r^{2(m_{k}-1)}.
\]

\end{proof}

We finish the proof by considering
 the last coordinate. Here, 
\[
\begin{split} 
r^{2(n_{{\N}} - 1)}\left| \left(f_{t} \right)_{{\N}}\right|^2 & = \sin^{2} ((n_{{\N}} - 1)t)\\
& \geq \frac{4}{\pi^{2}} (n_{{\N}} - 1)^{2} t^2
\end{split}
\]
for $t$ small enough.  

Now, choose $\xi$ so that $0 < |\xi| < 1$ and
\[\left( 1 - \frac{1}{\xi^2}\right) C''^2 \geq -\frac{1}{({\N}-1)\pi^2} (n_{{\N}}-1)^2.\]
Then for small $r$,
\[\begin{split} r^{2(n_{{\N}}-1)} |f_{t}|^2 & \geq \sum_{k=1}^{{\N}-1} r^{2(n_{{\N}}-n_{k})} \left[\left( 1- \frac{1}{\xi^2}\right) C''^2 t^2 + \left( 1- \xi^2 \right) D^2 r^{2(m_{k}-1)}\right]  + \frac{4}{\pi^2} (n_{{\N}}-1)^2 t^{2}\\
& \geq (1- \xi^2)D^2 r^{2n} + \frac{3}{\pi^2} (n_{{\N}}-1)^2 t^2
\end{split}\]
which is a good enough lower bound since the coeffficients are positive.
This concludes the proof of Lemma \ref{lemma:denom}.

\section{Case III: Different Top Exponents, $n_\N=1$}\label{subsec:case3}

This last section deals with the case that the end is given by coordinate 1-forms that are all holomorphic except for the last, which has a simple pole. The simplest case is the graph given by
\[
\left(\omega_1,\omega_2,\omega_3\right)=\left(1,i,\frac{1}{z}\right)dz
\]
which has a horn end of type $(0,0,1)$ at $0$. 
We will show that the Gauss curvature of a horn end is always $0$.  When $n_{\N-1}<n_\N$ and $n_\N\neq 1$, each singularity of the geodesic curvature integrand contributed $-\pi$ to the Gauss curvature.  The Gauss curvature behaves differently when $n_\N=1$.  Either the geodesic curvature integrand has no singularities and converges uniformly to $0$ as $r\rightarrow 0$ or the contributions from the singularities cancel.
\begin{example}
Before analyzing the total curvature when $n_\N=1$, consider the an example of a horn end of type $(-1,0,1)$ with figure eight cross sections, given by  
\[
\left(\omega_1,\omega_2,\omega_3\right)=\left(iz,i,\frac{1}{z}\right)dz \ .
\]
Applying the holomorphic diffeomorphism $\phi(z)=ize^z$ and an affine transformation to the above one-forms, we get the normalized one-forms
\[
\begin{split}
\psi_1&=\left(-iz-3iz^2-4iz^3+O(r^4)\right)dz\\
\psi_2&=\left(\frac{1}{z}-2z-\frac{3z^2}{2}+\frac{2z^3}{3}+O(r^4)\right)dz\\
\psi_3&=\left(\frac{1}{z}+1\right)dz \ .\\
\end{split}
\]
\begin{figure}[h]
	\centerline{ 
		\includegraphics[height=1.9in]{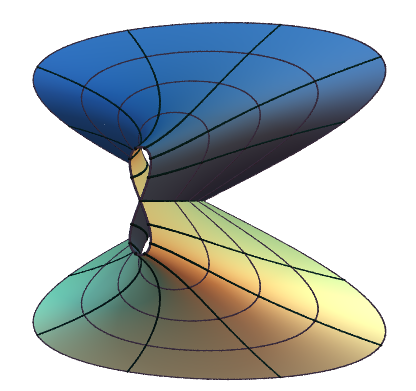}
	}
	\caption{Normalized end of type $(-1,0,1)$}
	\label{figure:case-101}
\end{figure}

The geodesic curvature integrand $\displaystyle \eta_r(t)$ is given by 
{\footnotesize
\[
\begin{split}
&\frac{r\left(4\sin^4{t}+\frac{r}{2}(8\cos{t}-9\cos(3t)+3\cos(5t))+O(r^2)\right)}{\sqrt{\sin^2{t}(1+8r\cos{t})+r^2(5-\frac{1}{2}\cos(2t)-\frac{5}{2}\cos(4t))+O(r^3)}\left(\sin^2(t)+\frac{r^2}{2}(5-3\cos(4t))+O(r^3)\right)}\ .\\
\end{split}
\]}
This is bounded for $r\to 0$ unless $\sin{t}=0$. Away from open neighborhoods of $\sin{t}=0$, the geodesic curvature integrand converges uniformly to 0 for $r\to 0$. 

\begin{figure}[h]
	\centerline{ 
		\includegraphics[height=1.9in]{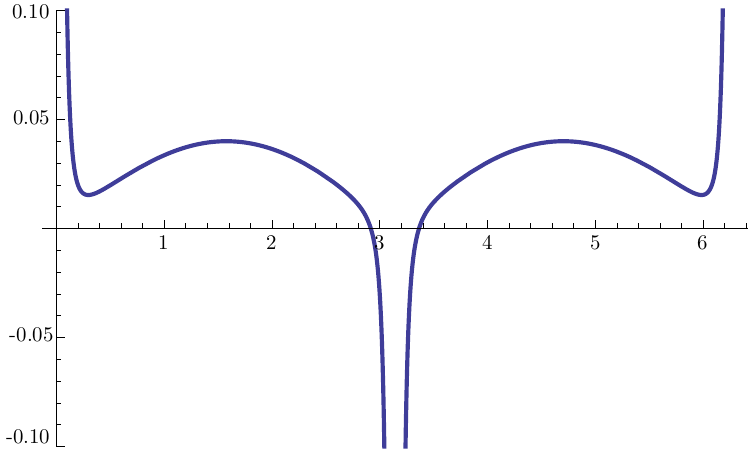}
	}
	\caption{Graph of $\eta_r(t)$ for small $r$}
	\label{figure:case-101b}
\end{figure}
If $\alpha\in\{0,\pi\}$ then 
\[
r\eta_r(rt+\alpha)=\begin{cases}\frac{1+O(r)}{\sqrt{2+t^2+O(r)}(1+t^2+O(r))},&\alpha=0\\-\frac{1+O(r)}{\sqrt{2+t^2+O(r)}(1+t^2+O(r))}&\alpha=\pi\\\end{cases}
\]
and
\[
\begin{split}
\lim_{r\rightarrow 0^+}\int_0^{2\pi}\eta_r(t)dt&=\sum_{\alpha}\lim_{r\rightarrow 0^+}\int_{\alpha-\epsilon}^{\alpha+\epsilon}\eta_r(t)dt\\
&=\sum_{\alpha}\lim_{r\rightarrow 0^+}\int_{-\epsilon}^{\epsilon}r\eta_r(r s+\alpha)ds\\
&=\sum_{\alpha}\lim_{r\rightarrow 0^+}\int_{-\infty}^{\infty}r\eta_r(r s+\alpha)\chi_{(-\epsilon/r,\epsilon/r)}(s)ds\\
&=\int_{-\infty}^{\infty}\frac{1}{\sqrt{2+s^2}(1+s^2)}ds+\int_{-\infty}^{\infty}\frac{-1}{\sqrt{2+s^2}(1+s^2)}ds\\
&=\frac{\pi}{2}-\frac{\pi}{2}\\
&=0 \ .
\end{split}
\]
The curvature contributions from the singularities are $\pi/2$ and $-\pi/2$, yielding the desired curvature of $0$.
\end{example}

We now return to the general discussion of the case that $n_\N=1$.  As in case II, after applying an orthogonal transformation, we can assume that $n_k<1$ for $1\leq k\leq \N-1$.  Then we can apply Proposition \ref{prop:coordinates} to normalize the last coordinate 1-form $\displaystyle\omega_\N(z)=\frac{1}{z}\,dz$.   Now, we  use an additional orthogonal transformation on the first $\N-1$ forms to further simplify the expression.  Consider the two vectors $v$ and $w$ that consist of the real and imaginary parts of the coefficients of terms of order $n_{\N-1}$.  If $v$ and $w$ span a plane (over $\mathbb{R}$) then we can rotate the surface so that the plane is parallel to the $x_{\N-2}x_{\N-1}$-plane, i.e. $n_{\N-3}<n_{\N-2}=n_{\N-1}$ with $t_{\N-2,1}\notin\{0,\pi\}$, and $t_{\N-1,1}\in\{0,\pi\}$.  Otherwise $v$ and $w$ are linearly dependent and we can rotate these vectors to be parallel to the $x_{\N-1}$-axis, i.e. $n_{\N-2}<n_{\N-1}$ with  $t_{\N-1,1}\in\{0,\pi\}$.  Note that to ensure $t_{\N-1,1} \in \{0, \pi\}$ in either case, we may need to apply a rotation in the domain but this will not affect $\omega_{\N}$.  
  
In this case, we will assume that the ${\N}$ coordinate 1-forms are given by
\[
\begin{split}
\omega_k(z)&=\sum_{j=1}^{\infty}r_{k,j}e^{it_{k,j}}z^{j-1-n_k}\,dz,\hspace{.1in}1\leq k\leq {\N}-1\\
\omega_{\N}(z)&=\frac{1}{z}\,dz\\
\end{split}
\]
with $n_{k-1}\leq n_{k}$ for $2\leq k\leq \N-2$, where either

\[
n_{\N-2}=n_{\N-1}<1,\hspace{.1in} t_{\N-2,1}\notin\{0,\pi\}, \hspace{.1in} t_{\N-1,1}\in\{0,\pi\}
\]
or
\[
n_{{\N}-2}<n_{\N-1}<1, \hspace{.1in} t_{\N-1,1}\in\{0,\pi\} \ .
\]
As in case II, there is a constant $M>0$ such that $r_{k,j}\leq M$ for $j\geq 1$ and $1\leq k\leq {\N}-2$. 

We have power series expansions
\[
f(r,t)=\left(\begin{split}
&\sum_{j=1}^{\infty}r^{j-n_1}\alpha_{1,j}(t)\\
&\hspace{.5in}.\\
&\hspace{.5in}.\\
&\hspace{.5in}.\\
&\sum_{j=1}^{\infty}r^{j-n_{{\N}-1}}\alpha_{{\N}-1,j}(t)\\
&\log{r}\\
\end{split}\right)
\]
where for $k=1,\ldots,{\N}-1$,
\[
\alpha_{k,j}(t)=\frac{r_{k,j}\cos{((j-n_k)t+t_{k,j})}}{j-n_k} \ .
\]

Here, we use the same definitions for $S_k$ and $T_k$, see Definition \ref{def:STR}.  The definition of $R_k$ is almost the same:
\begin{equation}
\label{eqn:STR}
\begin{split}
S_k(r,t)&=\sum_{j=1,j\neq n_k}^{\infty}r^{j-1}\alpha_{k,j}'(t)\\
T_k(r,t)&=-S_k'(r,t)=\sum_{j=1}^{\infty}r^{j-1}(j-n_k)^2\alpha_{k,j}(t)\\
R_k(r,t)&=\sum_{j=1}^{\infty}r^{j-1}(j-n_k)\alpha_{k,j}(t)\\
\end{split}
\end{equation}
Then we have by straightforward computation:

\begin{lemma}\label{lemma:gctermsnd1easy}
\[
\begin{split}
f_t&=r^{1-n_{\N-1}}\left(r^{n_{\N-1}-n_1}S_1(r,t),\ldots,r^{n_{\N-1}-n_{{\N}-2}}S_{{\N}-2}(r,t),S_{\N-1}(r,t),0\right)\\
f_{tt}&=-r^{1-n_{\N-1}}\left(r^{n_{\N-1}-n_1}T_1(r,t),\ldots,r^{n_{\N-1}-n_{{\N}-2}}T_{{\N}-2}(r,t),T_{\N-1}(r,t),0\right)\\
f_r&=\frac{1}{r}\left(r^{1-n_1}R_1(r,t),\ldots,r^{1-n_{\N-2}}R_{\N-2}(r,t),r^{1-n_{{\N}-1}}R_{\N-1}(r,t),1\right)\\
\eta_r(t)&=r\frac{\left(\sum r^{-2n_k}S_k^2\right)\left(\sum r^{-2n_k}R_kT_k\right)-\left(\sum r^{-2n_k}S_kR_k\right)\left(\sum r^{-2n_k}S_kT_k\right)}{\sqrt{\left(\sum r^{2(1-n_k)}R_k^2+1\right)\sum r^{-2n_k}S_k^2-r^{2(1+n_{\N-1})}\left(\sum r^{-2n_k}S_kR_k\right)^2}\sum r^{-2n_k}S_k^2}
\end{split}
\]
\end{lemma}

If $n_{\N-2}=n_{\N-1}<1,\hspace{.1in} t_{\N-2,1}\notin\{0,\pi\}$, and $t_{\N-1,1}\in\{0,\pi\}$ then there are no singularities for $\eta_r$ when $r=0$ since they would occur when 
\[
r_{\N-2,1}^2\sin^2((1-n_{\N-1})t+t_{\N-2,1})+r_{\N-1,1}^2\sin^2((1-n_{\N-1})t)=0
\]
which would force $t_{\N-2,1}\in\{0,\pi\}$.  Hence, $\eta_r(t)$ converges uniformly to $0$ as $r\rightarrow 0$, and
\[
\lim_{r\rightarrow 0^+}\int_0^{2\pi}\eta_r(t)\,dt=0 \ .
\]

If $n_{\N-2}<n_{\N-1}<1$  it is useful to apply another normalization, utilizing the holomorphic diffeomorphism $\phi(z)=ze^z$.  If $n_{\N-1}=0$ then it will be necessary to do an orthogonal transformation to ensure that $n_{\N-1}<0$.  We can do this since in this case, $\omega_{\N-1}(z)=(\pm 1+h.o.t)dz$ and $\omega_\N(z)=\left(\frac{1}{z}+1\right)dz$.  This gives us the one-forms

\[
\begin{split}
\omega_k(z)&=\sum_{j=1}^{\infty}r_{k,j}e^{it_{k,j}}z^{j-1-n_k}dz,\hspace{.1in}1\leq k\leq {\N}-1\\
\omega_{\N-1}(z)&=\left(\frac{1}{z}+\sum_{j=1}^{\infty}r_{\N-1,j}e^{it_{\N-1,j}}z^{j-1-n_{\N-1}}\right)dz\\
\omega_{\N}(z)&=\left(\frac{1}{z}+1\right)dz\\
\end{split}
\]
where $n_k\leq n_{k+1}$ for $k=1,2,\ldots,{\N}-2$ and $n_{{\N}-1}<0$.

Then we have power series expansions
\[
f(r,t)=\left(\begin{split}
&\sum_{j=1}^{\infty}r^{j-n_1}\alpha_{1,j}(t)\\
&\hspace{.5in}.\\
&\hspace{.5in}.\\
&\hspace{.5in}.\\
&\sum_{j=1}^{\infty}r^{j-n_{{\N}-2}}\alpha_{{\N}-2,j}(t)\\
&\log{r}+\sum_{j=1}^{\infty}r^{j-n_{{\N}-1}}\alpha_{{\N}-1,j}(t)\\
&\log{r}+r\cos{t}\\
\end{split}\right) \ .
\]

Again, using $R_k, S_k$, and $T_k$ as defined in \ref{eqn:STR}, we have by straightforward computation:

\begin{lemma}\label{lemma:gctermsnd1hard}
\[
\begin{split}
f_t&=r\left(r^{-n_1}S_1(r,t),\ldots,r^{-n_{{\N}-1}}S_{{\N}-1}(r,t),-\sin{t}\right)\\
f_{tt}&=-r\left(r^{-n_1}T_1(r,t),\ldots,r^{-n_{\N-1}}T_{{\N}-1}(r,t),\cos{t}\right)\\
f_r&=\frac{1}{r}\left(r^{1-n_1}R_1(r,t),\ldots,r^{1-n_{\N-2}}R_{\N-2}(r,t),1+r^{1-n_{\N-1}}R_{\N-1}(r,t),1+r\cos{t}\right) \ .\\
\end{split}
\]
\end{lemma}

Recall from Lemma \ref{lemma:gc2} that the geodesic curvature integrand is bounded above by
\[
|\eta_r(t)| \le  \frac{|f_t\wedge f_{tt}|}{ |f_t|^2}\ .
\]
By Lemma \ref{lemma:gctermsnd1hard}, this is bounded for $r\to 0$ unless $\sin{t}=0$. Moreover, the formula for $\eta_r(t)$ in Lemma \ref{lemma:gc1} together with Lemma \ref{lemma:gctermsnd1hard} shows that away from open neighborhoods of $\sin{t}=0$, the integrand converges uniformly to 0 for $r\to 0$. 

We will now analyze the singular behavior of the geodesic curvature integrand.  Suppose $\sin{t}=0$.  Then $\cos{t}\neq 0$, and again  Lemma \ref{lemma:gc1}  and Lemma \ref{lemma:gctermsnd1hard} imply that  our integrand does indeed have singularities at $0$ and $\pi$ when $r\to 0$.
We use the same blow-up argument from section \ref{subsec:blowup} to deal with the singularities in this case.  With this normalization, the curvature contribution from the singularities at $t=0$ and $t=\pi$ will have {\em opposite} signs and thus cancel. They can be computed as $\pi/2$ and $-\pi/2$, respectively, instead of a contribution of $-\pi$ from each singularity.

First, we deal with the blow-up near the singularity at $t=0$.  Using the slightly adjusted terms
\[
\begin{split}
n&=\min_k\{-n_k+m_k-1\}\\
k^*&=\{k\in\{1,\ldots,{\N}-1\}:-n_k+m_k-1=n\}\\
b&=\sum_{k\in k^*}r_{k,m_k}^2\sin^2(t_{k,m_k})\\
c&=r_{\N-1,m_{\N-1}}\sin(t_{\N-1,m_{\N-1}})\\
\end{split}
\]
 for the blow-up near the singularity at $t=0$, the terms from Lemma \ref{lemma:gctermsnd1hard} can be estimated as

\[
\begin{split}
S_k(r,r^nt)&=
\begin{cases}
-r^{m_k-1}r_{k,m_k}\sin{(t_{k,m_k})}+O(r^{m_k}), & m_k<\infty
\\
r_{k,1}(n_k-1)r^nt+O(r^{n+1}), & m_k=\infty
\end{cases}\\
T_k(r,r^nt)&=(1-n_k)r_{k,1}\cos(t_{k,1})+O(r)\\
R_k(r,r^nt)&=r_{k,1}\cos(t_{k,1})+O(r)\\
\cos{r^nt}&=1+O(r^{2n})\\
\sin{r^nt}&=r^nt+O(r^{n+1}) \ .\\
\end{split}
\]

Then
\[
\begin{split}
f_t\cdot f_t&=r^{2(n+1)}\left[b+t^2+O(r)\right]\\
f_r\cdot f_{tt}&=-1+O(r)\\
f_r\cdot f_t&=-r^n\left[r^{-n_{\N-1}+m_{\N-1}-1-n}c+t+O(r)\right]\\
f_{tt}\cdot f_t&=r^{2+n}\left[t+O(r)\right]\\
f_r\cdot f_r&=\frac{1}{r^{2}}\left[2+O(r)\right] \ .\\
\end{split}
\]
Combining everything gives
\[
(f_r\cdot f_t)(f_{tt}\cdot f_t)-(f_t\cdot f_t)(f_r\cdot f_{tt})=r^{2(n+1)}\left[b-r^{-n_{\N-1}+m_{\N-1}-1-n}ct+O(r)\right]
\]

\[
(f_r\cdot f_r)(f_t\cdot f_t)-(f_r\cdot f_t)^2=r^{2n}\left[2b-2r^{-n_{\N-1}+m_{\N-1}-1-n}ct-r^{2(-n_{\N-1}+m_{\N-1}-1-n}c^2+t^2+O(r)\right]
\]
and so
\[
r^n\eta(r^nt)=\frac{b-r^{-n_{\N-1}+m_{\N-1}-1-n}ct+O(r)}{\sqrt{2b-2r^{-n_{\N-1}+m_{\N-1}-1-n}ct-r^{2(-n_{\N-1}+m_{\N-1}-1-n}c^2+t^2+O(r)}(b+t^2+O(r))} \ .
\]
Hence,
\[
\lim_{r\rightarrow 0}r^n\eta(r^nt)=\begin{cases}
\frac{b-ct}{\sqrt{2b-2ct-c^2+t^2}(b+t^2)}, \hspace{.45in}-n_{\N-1}+m_{\N-1}-1=n
\\
\frac{b}{\sqrt{2b+t^2}(b+t^2)}, \hspace{.9in}-n_{\N-1}+m_{\N-1}-1>n
\end{cases} \ .\\
\]

When we do the blowup centered at $t=\pi$ instead of at $t=0$, after employing the substitution $t\rightarrow (-1)^nt+\pi$, we get  
\[
\lim_{r\rightarrow 0}r^n\eta(r^nt)=\begin{cases}
-\frac{b-ct}{\sqrt{2b-2ct-c^2+t^2}(b+t^2)}, \hspace{.45in}-n_{\N-1}+m_{\N-1}-1=n
\\
-\frac{b}{\sqrt{2b+t^2}(b+t^2)}, \hspace{.9in}-n_{\N-1}+m_{\N-1}-1>n
\end{cases} \ .\\
\]

We are in a setting in which we can directly apply the arguments in sections \ref{subsec:blowup},\ref{subsec:numerator}, and \ref{subsec:denominator} to show $\eta_r$ is uniformly bounded by a $L^1$ function.  So, the curvature contributions from the two singularities at $t=0$ and $t=\pi$ will cancel, giving zero curvature.

This concludes the proof in case III of Theorem \ref{thm:kappa}.


\bibliographystyle{plain}
\bibliography{bibliography}

\end{document}